\newif\ifGross

\Grosstrue
\Grossfalse

\ifGross
\documentclass[12pt]{article}
\usepackage{fullpage}
\else
\documentclass[12pt]{article}
\usepackage{fullpage}
\fi

\usepackage[color]{changebar}
\usepackage{amsfonts}
\usepackage{amsmath,amssymb}
\usepackage{bm}
\usepackage{mathtools}
\usepackage{graphicx}
\usepackage[ruled,linesnumbered,vlined]{algorithm2e}
\usepackage{multirow}
\usepackage{mathptmx}      
\usepackage{algorithm2e}

\usepackage{xspace}
\usepackage{makecell}
\usepackage{doi}
\usepackage{todonotes}

\usepackage{epstopdf} 
\usepackage{verbatim}
\usepackage{amsthm}
\usepackage{xspace}

\usepackage{dsfont}

\usepackage{pgfplots}
\pgfplotsset{compat=1.14}
\usepackage{subfig}

\usepackage{array}
\usepackage{tabu}
\usepackage{mathtools}
\usepackage{nicefrac}

\usepackage{tabularx,booktabs}

\usepackage{fix-cm}

\usepackage{wrapfig}
\newcolumntype{Y}{>{\centering\arraybackslash}X}

\newcommand{\E}{{\rm E}}
\newcommand{\eul}{{\rm e}}
\newcommand{\Var}{{\rm Var}}

\newcommand{\LP}[1]{\ensuremath{\mathcal{L}_{#1}}}
\newcommand{\LPnorm}[1]{\LP{#1}-norm}
\DeclareMathOperator{\di}{d\!}

\graphicspath{{./Figs/}}

\allowdisplaybreaks 

\newtheorem{theorem}{Theorem}
\newtheorem{definition}[theorem]{Definition}

\newtheorem{example}[theorem]{Example}

\usepackage{tikz}
\usetikzlibrary{shapes, arrows, positioning, backgrounds,calc,trees,fit,patterns}
\usetikzlibrary{decorations, decorations.text}
\usetikzlibrary{matrix,decorations.pathreplacing}
\usetikzlibrary{arrows.meta,backgrounds,shadings,shadows}

\newcommand{\real}{\ensuremath{\mathbb{R}}\xspace}

\newcommand{\R}{\real}

\allowdisplaybreaks 

\begin{document}

\title{The Convergence Indicator: Improved and completely characterized parameter bounds for actual convergence of Particle Swarm Optimization}

\author{Bernd Bassimir \qquad Alexander Ra\ss
\thanks{Corresponding author} \qquad Rolf Wanka
\\
Department of Computer Science\\
University of Erlangen-Nuremberg, Germany\\
{\sf\small $\{$bernd.bassimir, alexander.rass, rolf.wanka$\}$@fau.de}
}

\date{ }
\maketitle

\newcommand{\sep}{ $\cdot$ }
\begin{abstract}
\noindent

Particle Swarm Optimization (PSO) is a meta-heuristic for continuous
black-box optimization problems. In this paper we focus on the
convergence of the particle swarm, i.\,e., the exploitation phase of
the algorithm. We introduce a new convergence indicator that can be
used to calculate whether the particles will finally converge to a single
point or diverge. Using this convergence indicator we provide the actual
bounds completely characterizing parameter regions that lead to a converging swarm.
Our bounds extend the parameter regions where convergence is guaranteed compared to bounds induced by converging variance which are usually used in the literature.
To
evaluate our criterion we describe a numerical approximation
using cubic spline interpolation. Finally we provide
experiments showing that our concept, formulas and the resulting convergence
bounds represent the actual behavior of PSO.

\medskip

\end{abstract}

\noindent
This research did not receive any specific grant from funding agencies in the public, commercial, or not-for-profit sectors.

\section{Introduction}
Particle Swarm Optimization (PSO) is a nature-inspired meta-heuristic,
first introduced by Eberhart and Kennedy in the year 1995~\cite{ken_eb_1995}, which
mimics the behavior of bird flocks and fish swarms. It is designed for
so-called continuous black-box optimization problems, i.\,e.,
objective functions that are not explicitly provided in a
closed formula. The algorithm consists of a
set $\{1,\ldots,N\}$ of particles, the
individuals of the swarm, that move through the $D$-dimensional
continuous search space. Each particle has a position $x$ and a
velocity $v$ that are updated via Movement Equations~\eqref{eq:movementv}
and~\eqref{eq:movementx} (see below). Furthermore each
particle remembers its best position found so far, its local attractor
$l$, and can access the best position any particle has found, the
global attractor $g$. The swarm operates in discrete time steps, where
one iteration consists of one application of the movement equations
for each particle.

For each dimension $d\in \{1,\ldots,D\}$ particle $n\in\{1,\ldots,N\}$
performs its move in iteration $t+1$ based on the following movement
equations:
\begin{align}
  v_{t+1,n,d}&=\chi \cdot v_{t,n,d}+c_l\cdot r_{t+1,n,d}\cdot(l_{t,n,d}-x_{t,n,d})+c_g\cdot s_{t+1,n,d}\cdot (g_{t,n,d}-x_{t,n,d}) \label{eq:movementv}\\
  x_{t+1,n,d}&=x_{t,n,d}+v_{t+1,n,d}, \label{eq:movementx}
\end{align}
where $r_{t+1,n,d}$ and $s_{t+1,n,d}$ are random variables on $[0,1]$
independently and uniformly drawn in every iteration for every particle and dimension.
$\chi$, $c_l$ and $c_g$ are constant parameters chosen by the user.  After the particles
finish their moves the local and global attractors get possibly
updated if better positions were found.

Since its introduction in 1995 PSO has become popular not only among
computer scientists as it can be easily implemented and adapted to
different problems. Many different variants of PSO were introduced in the
literature and analyzed.

A special focus of the analysis of PSO is on the convergence and
stagnation, i.\,e., the premature convergence to non-optimal points,
of the swarm. PSO as a meta-heuristic should fulfill the two
principals of exploration, i.\,e., finding new good positions, and
exploitation, i.\,e., improving on already found points. The
parameters $\chi$, $c_l$ and $c_g$ have a major impact on the balance
between exploration and exploitation. Good choices for the parameters
should allow for extensive exploration, while the particles should
ultimately converge to the global attractor.

Various publications supply discussions on parameter settings and
bounds on them such that convergence can be guaranteed. Some of those which
analyze the quality and/or convergence guarantees of different
parameter settings theoretically are (among many others)
\cite{BE:02,Cleg15,Gazi12,Harrison2018,JLY:07a,JLY:07,Poli09,SW:15,T:03}.

In \cite{SW:15} Schmitt and Wanka
use drift theory to prove the convergence to local optima on
one-dimensional objective functions and introduce a slightly modified
PSO to extend this analysis to $D$-dimensional objective functions.

In \cite{BE:02} Van den Bergh and Engelbrecht introduce a modified PSO
algorithm that adds an adaptive noise term at each iteration. They
prove that this PSO has a guaranteed convergence.

In \cite{T:03}, Trelea could formulate bounds on %for
the parameters of PSO such that the expected value of the position of a particle
converges to the global attractor. The bounds on %for
the parameters presented in
\cite{T:03} are $\chi<1$, $c_l+c_g>0$ and $c_l+c_g<4\cdot(\chi+1)$ and these
bounds are visualized in Figure~\ref{fig:convergencebounds} if $c_l=c_g=c$ by
the triangle enclosed by the black straight line, the right bound of the graph
and the $\chi$-axis (i.\,e., the triangle described by the points $(-1,0)$, $(1,4)$, and $(1,0)$).
Please note that there are parameter configurations where the expected values of
the particles' positions converge but actually the particles' positions diverge
(area between red and black curve in Figure~\ref{fig:convergencebounds} where
the black curve is larger than the red curve). On the other hand, there are
also cases where the expected values of the particles' positions diverge but actually the particles'
positions converge (area between red and black curve in
  Figure~\ref{fig:convergencebounds} where the red curve is larger than the
black curve). Therefore convergence of the expected value is neither a necessary
condition nor a sufficient condition for convergence of the particles'
positions.
Trelea could identify several different
convergence behaviors that the swarm could exhibit based on the
parameter selection.

In \cite{JLY:07a,JLY:07}, Jiang et\,al.\,could prove parameter bounds
that lead to a
convergence of the variance and expected value of the particles'
positions given the attractors are constant. Especially if the attractors coincide then the variance converges to zero.
If $c_l=c_g=c$ then the area for parameters leading to a convergence of the
variance is bounded by the blue curve which represents the inequation
$c_l+c_g=2 c \leq 24(1-\chi^2)/(7-5\chi)$ and the $\chi$-axis in
Figure~\ref{fig:convergencebounds}. In addition to \cite{JLY:07a,JLY:07} this bound has been proposed in \cite{Poli09}.
Please note that this is a sufficient condition to guarantee convergence of the particles' positions but it is not a necessary condition.

In this paper we focus on the convergence of the swarm, i.\,e., convergence of the particles' positions, to the global
attractor. We %could theoretically
mathematically prove and experimentally confirm that there are parameter sets
where the expected value and the variance of the particles' positions
do not converge, however, the positions of the swarm converge with probability
one. Please note that this is not a contradiction, which we will explain %describe
in Section~\ref{sec:convergence_analysis} after Definition~\ref{def:convindicator}
in Example~\ref{ex:logisbetter} below.
In \cite{Harrison2018} the quality of different parameter settings is discussed
and various parameter sets are experimentally compared. There, good results
are also obtained while using parameter sets beyond the bound presented in
\cite{JLY:07a,JLY:07,Poli09}, where the variance of the particles' positions
diverges. This indicates that the bound for converging variance is not
the actual bound for the convergence of the swarm.
If attractors coincide and $c_l=c_g=c$ then the area for parameters leading to
a convergence of the particles' positions is bounded by the red curve and the
$\chi$-axis in Figure~\ref{fig:convergencebounds}. Please note that the slope
of the red curve directly to the left of $\chi=0$ is nearly twice as much as
the slope directly to the right of $\chi=0$, which is not a numerical error.
Consequently, it is an interesting observation itself, which could be analyzed
more in detail in a future work.
We introduce a new measure for the
convergence of the particles' positions and provide the actual parameter bounds that
lead to the convergence of the particles' positions. These bounds can be
obtained with arbitrary precision by numerical approximations using spline interpolation. For $\chi=0$ this bound
can also be obtained by closed theoretical analysis.
Hence, we extend the range of parameter settings which provably lead to convergence of PSO as specified in the following theorem.
\begin{theorem}{\color{white}{-}}\label{thm:figureClassification}
  \begin{itemize}
    \item The particles' positions converge if the parameters are chosen strictly between the red curve in Figure~\ref{fig:convergencebounds} and the $\chi$-axis.
    \item The particles' positions diverge if the parameters are chosen strictly outside the region between the red curve in Figure~\ref{fig:convergencebounds} and the $\chi$-axis.
  \end{itemize}
\end{theorem}

In Section~\ref{sec:convergence_analysis} we introduce our new
convergence indicator and calculate the difference between two
consecutive convergence indicators which can be used to prove bounds on
the parameters of the PSO algorithm such that convergence takes place. In
Section~\ref{sec:numeric}
we explain how these bounds can be calculated numerically and compare the
results with experiments confirming the results.

\begin{figure}[htb]
  \begin{center}
	\include{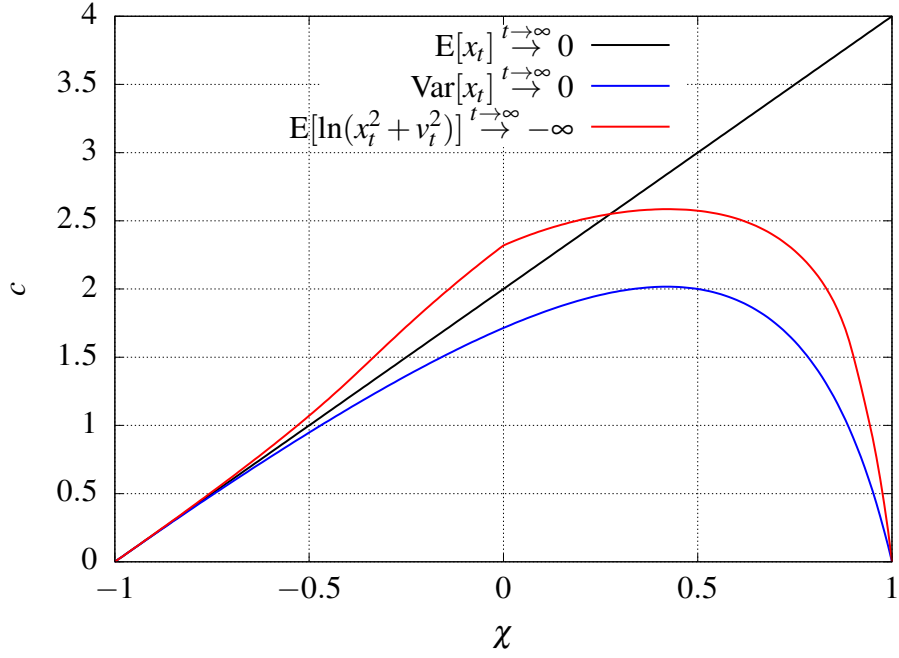}
  \end{center}
  \caption{Bounds for the convergence on the expected position to zero (see also \cite{T:03}), on variance of the position to zero (see also \cite{JLY:07a,JLY:07}) and on the expected logarithm of position and velocity to minus infinity \label{fig:convergencebounds}}
\end{figure}

\section{Theoretical Analysis on Convergence}
\label{sec:convergence_analysis}
First we introduce some simplifications. In practice -- at least if in some
sense convergence takes place -- the positions of the local attractors
and the global attractor converge to a single point $p$ in the search space. Then
for every $\varepsilon>0$ there is an iteration $t_0$ such that the positions of
the attractors have Euclidean distance less than $\varepsilon$ for all
iterations $t>t_0$. If convergence of all positions should take place then the
distances of the positions to $p$  have to decrease if they are currently of
significantly larger magnitude compared to $\varepsilon$. As this needs to happen for any $\varepsilon>0$ it is sufficient to check whether positions of the particles converge if attractors exactly have position $p$. Similar simplifications are commonly used in the literature (see \cite{T:03}).

Without loss of generality we are assuming that all entries of $p$ are zero and
therefore $l_{t,d,n}=g_{t,d,n}=0$.  Then the movement equations are independent
for each particle and each dimension.  Therefore we omit the indices $d$ for
the dimension and $n$ for the particle from now on and analyze only the
movement in a single dimension and for a single particle. The movement
equations in Equations~\eqref{eq:movementv} and~\eqref{eq:movementx} degenerate to
\begin{align}
  v_{t+1}&=\chi \cdot v_{t}-c_l\cdot r_{t+1}\cdot x_{t}-c_g\cdot s_{t+1}\cdot x_{t}=\chi \cdot v_{t}-H_{t+1} \cdot x_t \label{eq:movementv2}\\
  x_{t+1}&=x_{t}+v_{t+1}=\chi \cdot  v_{t}+(1-H_{t+1})\cdot x_t, \label{eq:movementx2}
\end{align}
where $r_{t+1}$ and $s_{t+1}$ are independent and uniformly distributed random
variables on $[0,1]$ and $H_{t+1}\sim c_l \cdot r_{t+1}+c_g \cdot s_{t+1}$.
PSO converges to $l_t=g_t=0$ iff $\lim\limits_{t\rightarrow\infty}{x_t}=0$ and $\lim\limits_{t\rightarrow\infty}{v_t}=0$.
This is equivalent to $\lim\limits_{t\rightarrow\infty}\ln\left(x_t^2+v_t^2\right)=-\infty$.

Therefore, we will use the following definition:
\begin{definition}[Convergence Indicator]
  \label{def:convindicator}
  We call $$\Phi_t:=\ln\left(x_t^2+v_t^2\right)$$ the \emph{convergence indicator}.
\end{definition}
With this convergence indicator we can say that PSO converges iff
$\lim\limits_{t\rightarrow\infty}\Phi_t=-\infty$.
This is the case iff on average the difference $\Phi_{t+1}-\Phi_{t}$ is negative.

To verify that the expected value and the variance of the position is not a good measure to quantify whether a random variable converges or not we discuss the following example.
\begin{example}
  \label{ex:logisbetter}
  Let $( X_t)_{t\geq 0}$ be a random process such that $X_0=1$,
  $\Pr[X_{t+1}=\eul\cdot X_t]=1/2$ and $\Pr[X_{t+1}=\eul^{-2}\cdot X_t]=1/2$ where
  $\eul$ refers to Euler's number ($\approx2.718$).
  Obviously for any $\varepsilon>0$ we have that
  $\lim_{t\rightarrow\infty}\Pr[X_t<\varepsilon]=1$. Nevertheless,
  $\E[X_{t}]=(\eul+\eul^{-2})/2\cdot\E[X_{t-1}]
  =\left((\eul+\eul^{-2})/2\right)^t\,
  \overset{t\rightarrow\infty}{\longrightarrow}\infty$
  and $\Var[X_t]\overset{t\rightarrow\infty}{\rightarrow}\infty$.
  In contrast if we analyze $Y_t:=\ln(X_t)$ we have $Y_0=0$,
  $\Pr[Y_{t+1}= Y_t+1]=1/2$ and $\Pr[Y_{t+1}=Y_t-2]=1/2$.
  Here $\E[Y_t]=(1-2)/2+\E[Y_{t-1}]
  =-t/2\overset{t\rightarrow\infty}{\rightarrow}-\infty$
  and $\E[Y_{t+1}-Y_t]=-1/2<0$.
  The drift of $Y_t$ to $-\infty$ and therefore also the tendency of $X_t$ to
  zero can easily be noticed in this logarithmic scale.
\end{example}

Analyzing expected values and variances of positions of particles is also not
helpful. Similarly to Example~\ref{ex:logisbetter} position and velocity of the
next iteration can be modeled as (matrix-) multiplication with values of the
previous iteration.
A better statistic for the analysis is the logarithm of the position. To avoid
numerical instabilities we use the square of the position and additionally add
the square of the velocity as the velocity also needs to converge to zero.

In the following sections we explain how the expected difference of consecutive
convergence indicators can be evaluated as if this difference is negative
then convergence can be expected.

\subsection{Main case: \texorpdfstring{$c_l\neq 0$ or $c_g\neq 0$}{at least one out of the two values cl and cg is not equal 0}}
\label{sec:maincase}
If both parameters $c_l$ and $c_g$ would be zero then PSO performs a deterministic movement independent of the objective function. As this is not profitable we exclude this case for our main analysis and move the respective discussions to Section~\ref{sec:c0}.

Similarly as in \cite{Ozcan99} we split the positions of the particles
into an angle and a magnitude.

Here we first describe how it is possible to define the current state by the
current position and some current angle instead of the current solution. We
present an iterative procedure to evaluate the distribution of this angle,
which is even independent of the current position. The distribution on this
angle enables us to evaluate the expected change in the convergence indicator
$\E[\Phi_{t+1}-\Phi_t]$.

If no atypical initialization is used then $x_t$ almost
surely~(in the mathematical sense, see~\cite{durrett2010probability})
has a value which is not \emph{exactly} equal to $0$, i.\,e., $\Pr[x_t \neq 0]=1$.
Excluding the event that $x_t=0$ for any $t$, we can specify $v_t$ by $x_t$
and an unique additional angle
$\alpha_t\in\left(\frac{-\pi}{2},\frac{\pi}{2}\right)$
such that $v_t=x_t\cdot\tan(\alpha_t)$.
Then $v_{t+1}=x_t\cdot(\chi\tan(\alpha_t)-H_{t+1})$ and
$x_{t+1}=v_{t+1}+x_t=x_t\cdot(\chi\tan(\alpha_t)+1-H_{t+1})$, where
$H_{t+1}=c_l r_{t+1}+c_g s_{t+1}$ and $s_{t+1},r_{t+1}$ are independent and uniformly
distributed over $[0,1]$ as already specified.
Therefore
\begin{equation}
  \tan(\alpha_{t+1})
=\frac{v_{t+1}}{x_{t+1}}
=\frac{\chi \tan(\alpha_t)-H_{t+1}}{\chi\tan(\alpha_t)+1-H_{t+1}}
=1-\frac{1}{1+\chi \tan(\alpha_t)-H_{t+1}}=f^+(\tan(\alpha_{t}),H_{t+1}),
  \label{eq:fplusA}
\end{equation}
where
\begin{equation}
f^+(m,h):=1-\frac{1}{1+\chi m-h}
  \label{eq:fplusB}
\end{equation}

This means that the new angle is only dependent on the old angle and the random
numbers of the current iteration but it is not dependent on the old position or
velocity. $\chi$ does not appear as parameter for $f^+$ as it is a constant.

For any $H_t=c_l\cdot r_t+ c_g\cdot s_t$ we can explicitly specify the probability density function.
Let $l_H:=\min(0,c_l)+\min(0,c_g)$ which is the \textbf{l}ower bound on the values which could be observed while evaluating $H_t$
and $u_H:=\max(0,c_l)+\max(0,c_g)$ which is the respective \textbf{u}pper bound.
Furthermore, let $c_{\min}:=\min(\vert c_l\vert, \vert c_g\vert)$ and
$c_{\max}:=\max(\vert c_l\vert, \vert c_g\vert)$.
Then the probability density function of $H_t$, where each $H_t$ is the
sum of two independent and uniformly distributed random variables, is
\begin{equation}
f_H(h):=\begin{cases}
  \frac{h-l_H}{\vert c_l\cdot c_g\vert} & \text{if }h\in (l_H,l_H+c_{\min})\\
  \frac{1}{c_{\max}} & \text{if }h\in [l_H+c_{\min},u_H-c_{\min}]\\
  \frac{u_H-h}{\vert c_l\cdot c_g\vert} & \text{if }h\in (u_H-c_{\min},u_H)\\
  0&\text{otherwise.}
\end{cases}
  \label{eq:densityH}
\end{equation}
Please note that this probability density function is also correct if one out of the two values $c_l$ and $c_g$ is zero.

Let $F_{\triangle,t}$ be the cumulative distribution function of
$\alpha_t$.
$F_{\triangle,t}$ is a continuous function for any $t>0$ as
\begin{align}
  &\lim_{\delta\rightarrow 0}\left(F_{\triangle,t+1}(\beta+\delta)-F_{\triangle,t+1}(\beta-\delta)\right)
  =\Pr[\alpha_{t+1}=\beta]
  \\&\overset{m=\tan(\beta)}{=}\int_{l_H}^{u_H}\Pr[f^+(\tan(\alpha_t),h)= m]f_H(h)\di h=0.
  \label{eq:probabilityexactvalue}
\end{align}
This integral has the value zero, because $f^+(m,h)$ is injective in respect to
$h$ and therefore the value $\Pr[f^+(\tan(\alpha_t),h)= m]$ can be positive (the value
can be at most one) only for a countably finite set of values for $h$, which
has no influence on the value of the integral.

For $t=0$ the cumulative distribution function $F_{\triangle,t}$ might not be a
continuous function as there exists for example the standard idea to use zero
as initial value for the velocity resulting in $\alpha_0=0$ deterministically.
This would result in $F_{\triangle,t}(\beta)=1$ if $\beta\geq 0$ and
$F_{\triangle,t}(\beta)=0$ otherwise, which is not continuous.

$F_{\triangle,t}$ is by the definition of cumulative distribution functions a monotonic function.
In combination with the property that $F_{\triangle,t}$ is continuous (for
$t>0$) we obtain by a theorem of Lebesgue (a proof can be found, e.\,g.,
in~\cite{B2003}) that $F_{\triangle,t}$ is
differentiable almost everywhere. We assume to have no pathological example
similar to the Cantor (ternary) function~\cite{cantor1884}.
Therefore it is assumed that this differentiability almost everywhere can be
extended to a derivative of $F_{\triangle,t}$, which will be denoted as
$f_{\triangle,t}$.
The function $f_{\triangle,t}$ is then the probability density function for $\alpha_t$ and any $t>0$.
In practice (see Section~\ref{sec:numeric}) this probability density function
exists and it is also a continuous function.

$F_{\triangle,t+1}$ can be calculated using $f^+$ from
Equation~\eqref{eq:fplusB} and in a second step $F_{\triangle,t}$ by the
following formula
\begin{equation}
  F_{\triangle,t+1}(\beta)=\Pr[\alpha_{t+1}\leq\beta]\overset{m=\tan(\beta)}{=}\int_{l_H}^{u_H}\Pr[f^+(\tan(\alpha_t),h)\leq m]f_H(h)\di h,
  \label{eq:alphadistributioniteration}
\end{equation}
where
\begin{align}&
  \Pr[f^+(\tan(\alpha_t),h)\leq m]=\nonumber\\&=\begin{cases}
	\begin{array}{@{}l@{} l l@{}}
	  1&=1&\text{if }\chi=0\wedge1-\frac{1}{1-h}\leq m\\
	  0&=0&\text{if }\chi=0\wedge1-\frac{1}{1-h}> m\\
	  \Pr[\alpha_t >  \gamma_1]&=1-F_{\triangle,t}(\gamma_1)&\text{if }m=1\wedge \chi>0\\
	  \Pr[\alpha_t <  \gamma_1]&=F_{\triangle,t}(\gamma_1)&\text{if }m=1\wedge \chi<0\\
	  \Pr[\alpha_t\not\in[\gamma_{\min},\gamma_{\max}]]&=1-F_{\triangle,t}(\gamma_{\max})+F_{\triangle,t}(\gamma_{\min})&\text{if }m>1\\
	  \Pr[\alpha_t\in[\gamma_{\min},\gamma_{\max}]]&=F_{\triangle,t}(\gamma_{\max})-F_{\triangle,t}(\gamma_{\min})&\text{if }m<1\\
  \end{array}
  \end{cases}
  \label{eq:probabilitypropagation}
\end{align}
and
\begin{align}&
  \gamma_1=\arctan\left(\frac{h-1}{\chi}\right)
  ,\quad \gamma_2=\arctan\left( \frac{h-1}{\chi}+\frac{1}{\chi\cdot(1-m)} \right)
  ,\\& \gamma_{\min}=\min(\gamma_1,\gamma_2)
  ,\quad \gamma_{\max}=\max(\gamma_1,\gamma_2).
  \label{eq:gammadefinitions}
\end{align}

This result can be obtained by the following rearrangements (* Please note that
for the final simplifications $\leq$ and $<$ are arbitrarily replaced as well
as $\geq$ and $>$, because the event that a single value is exactly achieved is
either zero in our setting (if $t>0$) or using a different value has no effect
as argued after Equation~\eqref{eq:probabilityexactvalue}.
Also we use that $\gamma_1\leq\gamma_2\Leftrightarrow \chi\cdot(1-m)>0$ in the final step):
\begin{align}
  &\Pr[f^+(\tan(\alpha_t),h)\leq m]\\&=\Pr\left[ 1-\frac{1}{1+\chi\tan(\alpha_t)-h}\leq m\right]\\
  &=\Pr\left[
	\begin{array}{l}
	  \left(\chi=0\wedge1-\frac{1}{1-h}\leq m\right)\\
	  \lor(1-m=0\wedge 0\leq \frac{1}{1+\chi\tan(\alpha_t)-h})\\
	  \lor\left(1-m>0\wedge 0\leq \frac{1}{1+\chi\tan(\alpha_t)-h} \wedge 1-m\leq\frac{1}{1+\chi\tan(\alpha_t)-h}\right)\\
	  \lor\left(1-m<0\wedge \left(0\leq \frac{1}{1+\chi\tan(\alpha_t)-h} \lor 1-m\leq\frac{1}{1+\chi\tan(\alpha_t)-h}\right)\right)
	\end{array}
  \right]\\
  &=\Pr\left[
	\begin{array}{l}
	  \left(\chi=0\wedge1-\frac{1}{1-h}\leq m\right)\\
	  \lor(1-m=0\wedge 0< {1+\chi\tan(\alpha_t)-h})\\
	  \lor\left(1-m>0\wedge 0< {1+\chi\tan(\alpha_t)-h} \wedge 1-m\leq\frac{1}{1+\chi\tan(\alpha_t)-h}\right)\\
	  \lor\left(1-m<0\wedge \left(0< {1+\chi\tan(\alpha_t)-h} \lor 1-m\leq\frac{1}{1+\chi\tan(\alpha_t)-h}\right)\right)
	\end{array}
  \right]\\
  &=\Pr\left[
	\begin{array}{l}
	  \left(\chi=0\wedge1-\frac{1}{1-h}\leq m\right)\\
	  \lor(1-m=0\wedge\chi>0\wedge \alpha_t>\gamma_1)\\
	  \lor(1-m=0\wedge\chi<0\wedge \alpha_t<\gamma_1)\\
	  \lor\left(1-m>0\wedge \chi>0\wedge \alpha_t>\gamma_1 \wedge \alpha_t\leq\gamma_2\right)\\
	  \lor\left(1-m>0\wedge \chi<0\wedge \alpha_t<\gamma_1 \wedge \alpha_t\geq\gamma_2\right)\\
      \lor\left(1-m<0\wedge \chi>0\wedge \left(\alpha_t>\gamma_1 \lor   \alpha_t\leq\gamma_2\right))\right)\\
      \lor\left(1-m<0\wedge \chi<0\wedge \left(\alpha_t<\gamma_1 \lor   \alpha_t\geq\gamma_2\right))\right)
	\end{array}
  \right]\\
  &\overset{*}{=}\Pr\left[
	\begin{array}{l}
	  \left(\chi=0\wedge1-\frac{1}{1-h}\leq m\right)\\
	  \lor(1-m=0\wedge\chi>0\wedge \alpha_t>\gamma_1)\\
	  \lor(1-m=0\wedge\chi<0\wedge \alpha_t<\gamma_1)\\
	  \lor\left(1-m>0\wedge \gamma_{\min}<\alpha_t<\gamma_{\max}\right)\\
	  \lor\left(1-m<0\wedge \left(\alpha_t<\gamma_{\min}\lor \gamma_{\max}<\alpha_t\right)\right)
	\end{array}
  \right]
  \label{eq:rearrangements}
\end{align}

Finally, we return to our initial aim to determine whether the convergence indicator converges to $-\infty$.
The differences of subsequent convergence indicators are only
dependent on the current angle, which we can see in the following series of
equations
\begin{align}
  &\E[\Phi_{t+1}-\Phi_t]=\E\left[\ln\left(x_{t+1}^2+v_{t+1}^2\right)-\ln\left({x_t^2+v_t^2}\right)\right]
=\E\left[\ln\left(\frac{x_{t+1}^2+v_{t+1}^2}{x_t^2+v_t^2}\right)\right]
\\
=&\E\left[\ln\left(\frac{ (\chi\tan(\alpha_t)-H_{t+1})^2 + (\chi\tan(\alpha_t)-H_{t+1}+1)^2}{1+\tan(\alpha_t)^2}\right)\right]
\\
=&\int_{-\frac{\pi}{2}}^{\frac{\pi}{2}}\E\left[\ln\left(\frac{ (\chi\tan(\alpha)-H_{t+1})^2+(\chi\tan(\alpha)-H_{t+1}+1)^2}{1+\tan(\alpha)^2}\right)\right]f_{\triangle,t}(\alpha)\di\alpha,\label{eq:expectedxvdifference}
\end{align}
where $f_{\triangle,t}$ is the probability density function for $\alpha_t$.

With the helping functions
\begin{equation}
  g(\alpha,h):=\ln\left(\frac{(\chi\tan(\alpha)-h)^2+(\chi\tan(\alpha)-h+1)^2}{1+\tan(\alpha)^2}\right)
  \quad
 f^{\emph{\scriptsize{in}}}(\alpha)=\int_{l_H}^{u_H}g(\alpha,h)\cdot f_H(h)\di h
  \label{eq:helpingfunction}
\end{equation}
we can shorten Equation~\eqref{eq:expectedxvdifference}. If we additionally
replace the expected value by the respective integral we obtain
\begin{equation}
  \eqref{eq:expectedxvdifference}=
  \int_{-\frac{\pi}{2}}^{\frac{\pi}{2}}\left(\int_{l_H}^{u_H}g(\alpha,h)\cdot f_H(h)\di h \right)f_{\triangle,t}(\alpha)\di \alpha
  =\int_{-\frac{\pi}{2}}^{\frac{\pi}{2}} f^{\emph{\scriptsize{in}}}(\alpha)\cdot f_{\triangle,t}(\alpha)\di \alpha .
  \label{eq:shorterexpectedxvdifference}
\end{equation}
In this final equation only $f_{\triangle,t}$ is changing for varying values of $t$.
The corresponding cumulative distribution function $F_{\triangle,t}$ on the angles $\alpha_t$ converge to a
unique stationary cumulative distribution function $F_{\triangle,\infty}$.
The reason for this cumulative distribution function to be unique is that for
the underlying process and for any starting angle $\beta$ the set of reachable
angles
$A_{\alpha_0=\beta,t}:=\lbrace \alpha\in(-\pi/2,\pi/2)\mid
\forall\varepsilon>0:
\Pr[\alpha_t\in(\alpha-\varepsilon,\alpha+\varepsilon)\mid \alpha_0=\beta]>0\rbrace$
converges to exactly the same set if $t$ tends to infinity.
If, e.\,g., $(0<\chi<1\wedge \max(c_l,c_g)>1)$, which is true for
many common parameter sets, this limit set contains all possible angles in
$(-\pi/2,\pi/2)$ as already $A_{\alpha_0=\beta,6}=(-\pi/2,\pi/2)$ for any $\beta\in(-\pi/2,\pi/2)$. In three
steps one can move from any starting angle to any open set containing angle
zero with positive probability.  In three further steps one can move with
positive probability from there to any other open set of angles.
To make this visible we define the function
$$f^-(m,h):=\frac{1}{\chi\cdot(1-m)}+\frac{h-1}{\chi}\enspace.$$
This function is the inverse function of $f^+$ in respect to $m$, i.\,e., $f^-(f^+(m,h),h)=m$ and also $f^-(\tan(\alpha_{t+1}),H_{t+1})=\tan(\alpha_t)$.
Then we can observe that the preimage of the angle zero is at least the set $[0,\pi/4]$ as $[0,1]\subset
f^-([0,0]\times(l_H,u_H))$ and $\arctan(1)=\pi/4$,
the preimage of the angles in $[0,\pi/4]$ is at least the set $[0,\pi/2)$ as
$[0,+\infty)\subset f^-([0,1]\times(l_H,u_H))$ and the preimage of the angles $[0,\pi/2)$ is the set of all possible angles in $(-\pi/2,\pi/2)$ as
$\R= f^-([0,+\infty)\times(l_H,u_H))$.
As $H_t$ has a probability density function which is positive in the interval
$(l_H,u_H)$, this justifies the claim that one can reach any open set
containing the angle zero with positive probability within exactly three
steps.
Additionally, from angle zero one can reach at least the angles in $(-\pi/2,0)$
as $(-\infty,0)\subset f^+([0,0]\times(l_H,u_H))$, from the angles $(-\pi/2,0)$
one can reach at least the angles in
$\left( (-\pi/2,0)\cup(\pi/4,\pi/2)\right)$ as
$\left((-\infty,0)\cup(1,+\infty)\right)\subset f^+( (-\infty,0)\times(l_H,u_H))$ and
from the angles in $\left( (-\pi/2,0)\cup(\pi/4,\pi/2)\right)$ one can reach all possible angles except exactly $\pi/4=\arctan(1)$ as
$(\R\setminus\lbrace 1\rbrace)= f^+(\left((-\infty,0)\cup(1,+\infty)\right) \times(l_H,u_H))$.
This justifies that one can reach any open set of angles from any open set containing angle zero with positive probability within exactly three steps.

Using
the associated stationary probability density function $f_{\triangle,\infty}$ in
Equation~\eqref{eq:shorterexpectedxvdifference} answers the question whether the particles converge.
\begin{theorem}
  \label{thm:classification}
  Let $\omega$ be the respective result of Equation~\eqref{eq:shorterexpectedxvdifference} using $f_{\triangle,\infty}$ as density function:
  \begin{equation*}
    \omega=\int_{-\frac{\pi}{2}}^{\frac{\pi}{2}} f^{\emph{\scriptsize{in}}}(\alpha)\cdot f_{\triangle,\infty}(\alpha)\di \alpha\enspace.
  \end{equation*}
  \begin{itemize}
    \item If $\omega<0$ then the particles position and velocity converge to the constant attractors.
    \item If $\omega>0$ then the particles position and velocity do not converge to the constant attractors.
  \end{itemize}
\end{theorem}
\begin{proof}
For any $\varepsilon>0$ there exists an iteration $t_0$ such that
$\vert\E[\Phi_{t+1}-\Phi_t]-\omega\vert<\varepsilon$ for all $t\geq t_0$ which
implies that
$\E[\Phi_t]
=\E[\Phi_0]+\sum_{i=0}^{t-1}\E[\Phi_{i+1}-\Phi_i]
\sim\omega\cdot t$.
Furthermore, the distribution of $(\Phi_t-\E[\Phi_t])/\sqrt{\Var[\Phi_t]}$
converges to a Gaussian/normal distribution as it can be written as sum of
consecutive convergence indicators and the covariance between
$\E[\Phi_{t_1+1}-\Phi_{t_1}]$ and $\E[\Phi_{t_2+1}-\Phi_{t_2}]$ tends to zero
if $\vert t_2-t_1\vert$ tends to $\infty$ (similar arguments apply as for the
reasoning on a unique limit distribution on $\alpha_t$).
The vanishing covariance also implies that the variance of $\Phi_t$ grows only linearly.
Therefore
$\E[\Phi_t]/\sqrt{\Var[\Phi_t]}$ grows or decreases in the order of $\sqrt{t}$
if $\omega>0$ or $\omega<0$ respectively.

These properties imply that if $\omega$ is negative then the probability
$\lim_{t\rightarrow\infty}\Pr[\Phi_t >\delta]=0$ for any value $\delta\in\R$.
Consequently the position and the velocity needs to converge to zero.

Analogously, if
$\omega$ is positive the particles positions and velocities will finally diverge.
\end{proof}
If the result $\omega$ evaluates to zero exactly then it has to be decided by different
means whether the particles do converge, do not converge or stay in some range
as also the initialization can have
an effect on this question in this case. Parameters, where the result is
exactly zero, are not recommended as also in cases where convergence still
appears, the time until sufficient optimization is accomplished would be quite
large as the drift of $\Phi_{t}$ to $-\infty$ is very slow and therefore also
the convergence of the position to the surroundings of the attractors.

Please note that by Theorem~\ref{thm:classification} also
Theorem~\ref{thm:figureClassification} is proven as $\omega$ is strictly less
than zero strictly between the red curve and the $\chi$-axis and $\omega$ is
strictly greater than zero strictly outside that region.

Theoretically at least the inner integral of
Equation~\eqref{eq:shorterexpectedxvdifference} can be solved analytically but
already displaying the respective solution, which can be obtained by maple or
other computer algebra systems, would require more than a page full of
formulas.
The question how this formula can be evaluated in practice is answered in Section~\ref{sec:numeric}.

\subsection{Special Case: \texorpdfstring{$c_l=c_g=0$}{cl and cg equal 0}}
\label{sec:c0}
In this atypical case where $c_l=c_g=0$ the velocity deterministically evaluates to $v_t=v_0\cdot\chi^t$ and therefore
$$
x_t=x_0+\sum_{i=1}^{t}v_i=x_0+v_0\cdot\chi\sum_{i=0}^{t-1}\chi^i=\begin{cases}
  x_0+v_0\cdot\chi\cdot\frac{1-\chi^t}{1-\chi}&\text{if }\chi\neq 1\\
  x_0+v_0\cdot t &\text{otherwise.}
\end{cases}
$$
Consequently, the position converges to some point iff
$\lim\limits_{t\rightarrow\infty}\chi^t=0\Leftrightarrow \vert \chi\vert <1$.
Whether this point is equal to the local and global attractor or even a local
or global optimum depends only on the initial choice of $x_0$ and $v_0$ and the
objective function, but for most settings the answer will very likely be ``no''.

For all other sections we assume that either $c_l$ or $c_g$ is not equal to zero.

\subsection{Special case: \texorpdfstring{$\chi=0$}{chi equals 0}}
Usually the expected differences of consecutive convergence indicators can not be presented as closed expressions.
For $\chi=0$ there is an exception and we can deduce the closed formula by an approach different from the previous sections.
First we describe the positions and velocities of iteration $t$ and $t+1$ by the position $x_{t-1}$.
$$\begin{array}{r@{} l r@{}l}
v_t&=-H_t\cdot x_{t-1},& x_t&=(1-H_t)\cdot x_{t-1},\\ 
v_{t+1}&=-(1-H_t)\cdot H_{t+1}\cdot x_{t-1},& x_{t+1}&=(1-H_t)\cdot (1-H_{t+1})\cdot x_{t-1}, 
\end{array}$$
This can be used to obtain the predicted closed formula in the case where $c_l\neq 0\neq c_g$:
\begin{align}
  &&&\E[\Phi_{t+1}-\Phi_t]=\E\left[\ln\left(x_{t+1}^2+v_{t+1}^2\right)-\ln\left(x_t^2+v_t^2\right)\right]\\
  =&&&\E\left[\ln\left(\frac{x_{t-1}^2\cdot\left(1-H_t\right)^2\cdot\left(H_{t+1}^2+(1-H_{t+1})^2\right)}{x_{t-1}^2\cdot\left( (1-H_t)^2+(H_t)^2\right)}\right)\right]\\
  =&&&\E\left[\ln\left(\left(1-H_t\right)^2\right) +\ln\left( H_{t+1}^2+(1-H_{t+1})^2\right) - \ln\left((1-H_t)^2+(H_t)^2\right) \right]\\
  \overset{\mathclap{H_t\sim H_{t+1}}}{=}&&&\E\left[\ln\left(\left(1-H_t\right)^2\right) \right]
  =\E\left[\ln\left(\left(1-c_l r_{t}-c_g s_{t}\right)^2\right)\right]\label{eq:expectedvaluechi0}\\
  =&&&\int_{0}^{1}\int_{0}^{1}\ln( ( 1-c_l r-c_g s)^2)\di r \di s\\
  =&&&\int_{0}^{1}\left[\left.\frac{1-c_l r-c_g s}{-c_l}\ln( ( 1-c_l r-c_g s)^2)-2 r\right\vert_0^1 \right]\di s\\
  =&&\frac{1}{c_l}&\int_{0}^{1}\left[-(1-c_l -c_g s)\ln( ( 1-c_l -c_g s)^2)-2 c_l+(1-c_g s)\ln( ( 1-c_g s)^2) \right]\di s\\
  =&&\frac{1}{c_l}&\left.\left[-\frac{(1-c_l -c_g s)^2}{-2 c_g}\left(\ln( ( 1-c_l -c_g s)^2)-1\right)\right.\right.\\&&&\left.\left.-2 c_l s+\frac{(1-c_g s)^2}{-2 c_g}\left(\ln( ( 1-c_g s)^2)-1\right) \right]\right\vert_0^1\\
	  =&&\frac{1}{2 c_l c_g}&\left[{(1-c_l -c_g)^2}\left(\ln( ( 1-c_l -c_g)^2)-1\right)-4 c_l c_g-{(1-c_g)^2}\left(\ln( ( 1-c_g)^2)-1\right)\right.
	  \nonumber\\&&&\left.-{(1-c_l)^2}\left(\ln( ( 1-c_l)^2)-1\right) -1\right]\\
      =&&\frac{1}{2 c_l c_g}&\left[{(1-c_l -c_g)^2}\ln( ( 1-c_l -c_g)^2)-{(1-c_g)^2}\ln( ( 1-c_g)^2)\right.\nonumber\\&&&\left.-{(1-c_l)^2}\ln( ( 1-c_l)^2)\right]-{3} \label{eq:chi0result}
\end{align}
and if $c=c_l=c_g$ then this equation simplifies to
\begin{equation}
	  \frac{1}{2 c^2}\left[{(1-2c)^2}\ln( ( 1-2c)^2)-2{(1-c)^2}\ln( ( 1-c)^2)\right]-{3}\enspace.
  \label{eq:clcgequal}
\end{equation}

For $c=c_l=c_g$ this expression is negative iff $c\in(0,\lambda)$ where $\lambda\approx 2.3195565$.

If additionally to $\chi=0$ also one out of the two values $c_l$ and $c_g$ is zero we obtain a slightly different result.
If $c_l=0$, $c_g=c$ or $c_l=c$, $c_g=0$ we can start with Expression~\eqref{eq:expectedvaluechi0} to obtain
\begin{align}
  \eqref{eq:expectedvaluechi0}&=\int_{0}^{1}\ln\left( \left(1-c\cdot r\right)^2\right)\di r
=\left.\left[\frac{1-c\cdot r}{-c}\ln\left( ( 1-c\cdot r)^2 \right)-2 r\right]\right\vert_0^1
  \\&=-\frac{1-c}{c}\ln\left( ( 1-c)^2 \right)-1\enspace.
\end{align}
This result can %would
also be obtained if one calculates the limit of the result in Equation~\eqref{eq:chi0result} if $c_l$ or $c_g$ tends to zero.

\section{Numeric Identification of Convergence}
\label{sec:numeric}

In this section we explain how the expected difference of consecutive convergence indicators can be evaluated in
practice and we present results on evaluations. Finally we compare the obtained
results with results from experiments.

A promising approach to solve this problem is cubic spline interpolation~\cite{HALL76}.
By cubic spline interpolation a function which is available at some reference points/knots is interpolated with piecewise cubic polynomials.

\subsection{Calculate the cumulative distribution function}
\label{sec:calcnextcdf}
We use spline interpolation to calculate the cumulative distribution function $F_{\triangle,t}$ for the
angle $\alpha_t$. As we are mainly interested in the unique limit distribution $F_{\triangle,\infty}$ the start distribution is not important. We use the uniform distribution on the angles in the range $[-\pi/2,\pi/2]$ as starting distribution $F_{\triangle,0}(\beta)=(\beta+\pi/2)/\pi$.
To calculate the next cumulative distribution function $F_{\triangle,t+1}$ by $F_{\triangle,t}$ we evaluate $F_{\triangle,t+1}$ at some set of reference knots $\lbrace\beta_1,\beta_2,\ldots,\beta_N\rbrace$ by Equation~\eqref{eq:alphadistributioniteration}, where $\beta_1=-\pi/2$, $\beta_N=\pi/2$ and for all $i$ we have $\beta_i<\beta_{i+1}$.
To achieve this we evaluate for fixed $\beta_i$ the integral in
Equation~\eqref{eq:alphadistributioniteration} again by spline interpolation
with a set of reference knots $\lbrace h_1,h_2,\ldots h_M\rbrace$, where
$h_1=l_H$, $h_M=u_H$ and for all $j$ we have $h_j<h_{j+1}$. For each $h_j$ and
fixed $\beta_i$ the expression
$\Pr[f^+(\tan(\alpha_t),h_j)\leq\tan(\beta_i)]\cdot f_H(h_j)$ can be evaluated
directly by Equation~\eqref{eq:densityH} in combination with
Equation~\eqref{eq:probabilitypropagation} as $F_{\triangle,t}$ is already
available - at least as an approximation. The respective integral of this expression can then easily be
evaluated through the integral on the respective spline interpolation. Finally,
we have the pairs $\left(\beta_i,F_{\triangle,t+1}(\beta_i)\right)$ and
therefore we can use them to calculate the next cumulative distribution
function $F_{\triangle,t+1}$ as spline interpolation.
By this procedure we can iteratively calculate spline interpolations on the cumulative distribution function.

We can stop doing further iterations if the difference of two consecutive
cumulative distribution functions is negligible as then the result is close to
the unique limit distribution and this distribution can then be used as
approximation on $F_{\triangle,\infty}$ and all $F_{\triangle,t}$ for larger
values of $t$.

\subsection{Calculate the expected difference of two consecutive convergence indicators}
\label{sec:calcfinal}
For the evaluation of Equation~\eqref{eq:shorterexpectedxvdifference} we use
again spline interpolation twice. For the inner integral
$f^{\emph{\scriptsize{in}}}(\alpha)=\int_{l_H}^{u_H}g(\alpha,h)\cdot f_H(h)\di h$
one can evaluate $g(\alpha,h)\cdot f_H(h)$ exactly and we approximate this
product by a spline through reference knots
$\lbrace h_1,h_2,\ldots, h_M\rbrace$, where $h_1=l_H$, $h_M=u_H$ and for all $j$
we have $h_j<h_{j+1}$. The integral on this spline can then easily be calculated.

By this helping procedure and the probability density function
$f_{\triangle,t}$ at iteration $t$, which can be obtained as described in the
previous section, we can also evaluate the complete integral displayed in
Equation~\ref{eq:shorterexpectedxvdifference}. For this purpose we evaluate
$f^{\emph{\scriptsize{in}}}(\alpha)\cdot f_{\triangle,t}(\alpha)$ for reference
knots $\lbrace\beta_1,\beta_2,\ldots,\beta_N\rbrace$, where $\beta_1=-\pi/2$,
$\beta_N=\pi/2$ and for all $i$ we have $\beta_i<\beta_{i+1}$. The final result
is then obtained by the evaluation of the integral on the corresponding spline.

As we can also use an approximation on $f_{\triangle,\infty}$ instead of $f_{\triangle,t}$ we can also evaluate the limit of expected difference of consecutive convergence indicators $\lim_{t\rightarrow\infty}\E[\Phi_{t+1}-\Phi_t]$.

\subsection{Notes on the numerical evaluations}
\begin{figure}[tb]
  \vspace{-30pt}
  \begin{center}
  \include{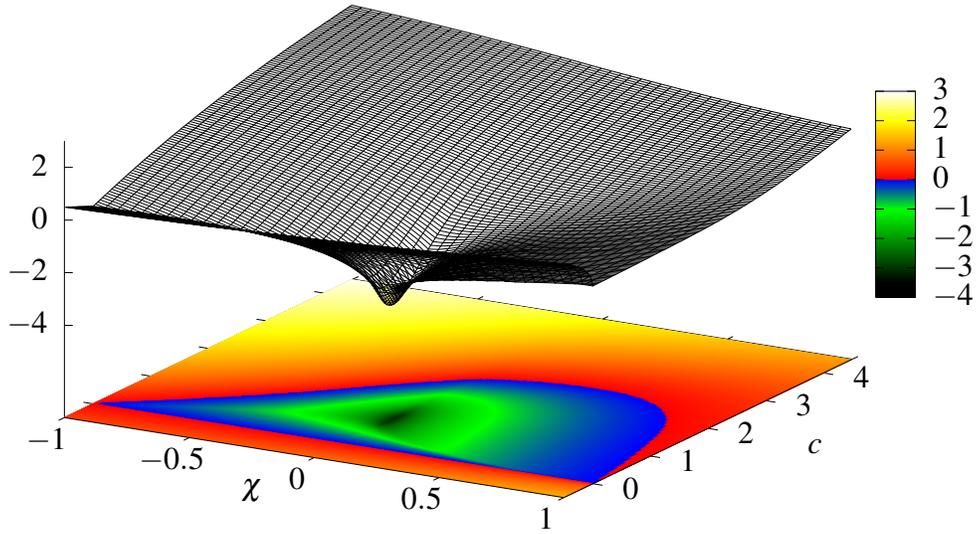}
\end{center}
\vspace{-35pt}
\caption{Expected value of $\Phi_{t+1}-\Phi_t$ if $t$ tends to infinity.}
\label{fig:conv3d}
\end{figure}

\begin{table}
  \centering
  \begin{tabular}{c c c c c}
    Reference &  $\chi$ & $c_l$ & $c_g$ & $\lim\limits_{t\rightarrow \infty}\E\left[\Phi_{t+1}-\Phi_t\right]$ \\
	% last value: measured error
\hline
\cite{CK:02,ES:00,T:03}& $0.72984$ & $1.496172$ & $1.496172$ & $-0.194063$\\
\cite{CD:01}& $0.72984$ & $2.04355$  & $0.94879$  & $-0.177108$\\
\cite{T:03}& $0.6$     & $1.7$      & $1.7$      & $-0.327742$\\
\cite{T:03}& $0.9$     & $0.1$      & $0.1$      & $-0.100728$\\
\cite{T:03}& $0.7$     & $0.3$      & $0.3$      & $-0.338770$\\
\cite{T:03}& $0.9$     & $3$        & $3$        & $ 0.380623$\\
\cite{T:03}& $0.1$     & $0.1$      & $0.1$      & $-0.241938$\\
\cite{T:03}& $0.1$     & $2.1$      & $2.1$      & $-0.485162$\\
\cite{T:03}& $-0.7$    & $0.5$      & $0.5$      & $-0.133533$
 \end{tabular}
 \caption{Standard parameter sets and corresponding expected change in the logarithm of position and velocity \label{tab:establishedparameters}}
\end{table}

\begin{figure}
     \captionsetup[subfigure]{justification=centering}
  \subfloat[$\chi\approx 0.730$, $c_l\approx 1.496$, $c_g\approx 1.496$ (almost identical to $\chi\approx 0.730$, $c_l\approx 2.044$, $c_g\approx 0.949$)][$\chi\approx 0.730$, $c_l\approx 1.496$, $c_g\approx 1.496$\newline(almost identical to $\chi\approx 0.730$, $c_l\approx 2.044$, $c_g\approx 0.949$)]
  {
	\begin{minipage}{0.48\textwidth}
	  \centering
	  \include{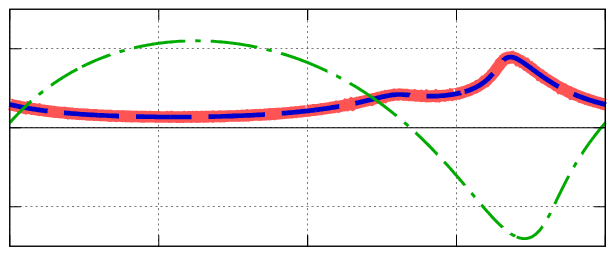}
	  \label{subfig:chi0-72984-c1-1-496172-c2-1-496172}
	\end{minipage}
  }
  \subfloat[$\chi=-0.7$, $c_l=0.5$, $c_g=0.5$]
  {
	\begin{minipage}{0.48\textwidth}
	  \centering
	  \include{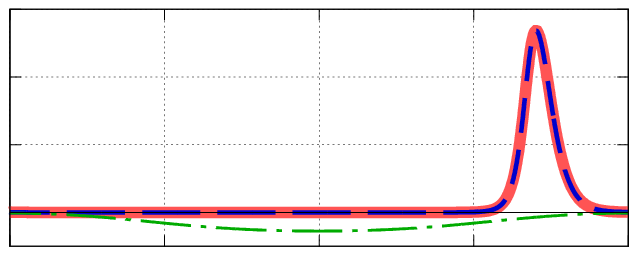}
	  \label{subfig:chi-0-7-c1-0-5-c2-0-5}
	\end{minipage}
  }\\
  \subfloat[$\chi= 0.6$, $c_l= 1.7$, $c_g= 1.7$]
  {
	\begin{minipage}{0.48\textwidth}
	  \centering
	  \include{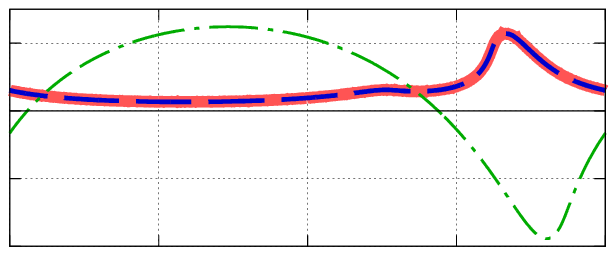}
	  \label{subfig:chi0-6-c1-1-7-c2-1-7}
	\end{minipage}
  }
  \subfloat[$\chi=0.1 $, $c_l=2.1 $, $c_g=2.1 $]
  {
	\begin{minipage}{0.48\textwidth}
	  \centering
	  \include{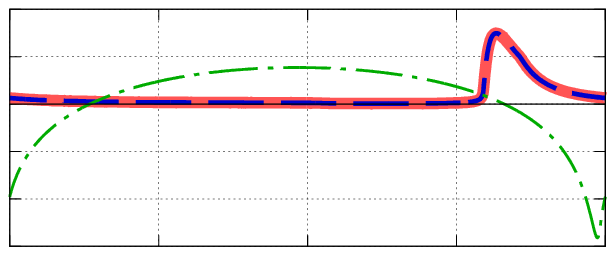}
	  \label{subfig:chi0-1-c1-2-1-c2-2-1}
	\end{minipage}
  }\\
  \subfloat[$\chi= 0.1$, $c_l= 0.1$, $c_g= 0.1$]
  {
	\begin{minipage}{0.48\textwidth}
	  \centering
	  \include{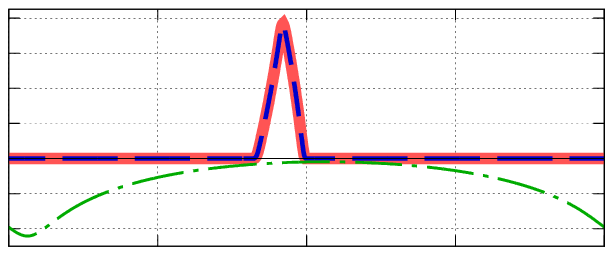}
	  \label{subfig:chi0-1-c1-0-1-c2-0-1}
	\end{minipage}
  }
  \subfloat[$\chi= 0.7$, $c_l= 0.3$, $c_g= 0.3$]
  {
	\begin{minipage}{0.48\textwidth}
	  \centering
	  \include{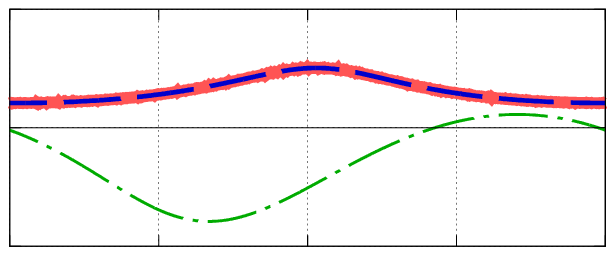}
	  \label{subfig:chi0-7-c1-0-3-c2-0-3}
	\end{minipage}
  }\\
  \subfloat[$\chi= 0.9$, $c_l= 0.1$, $c_g= 0.1$]
  {
	\begin{minipage}{0.48\textwidth}
	  \centering
	  \include{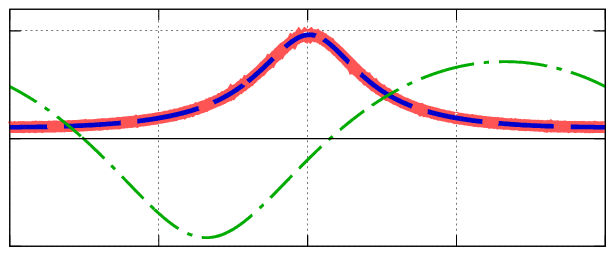}
	  \label{subfig:chi0-9-c1-0-1-c2-0-1}
	\end{minipage}
  }
  \subfloat[$\chi= 0.9$, $c_l=3 $, $c_g=3 $]
  {
	\begin{minipage}{0.48\textwidth}
	  \centering
	  \include{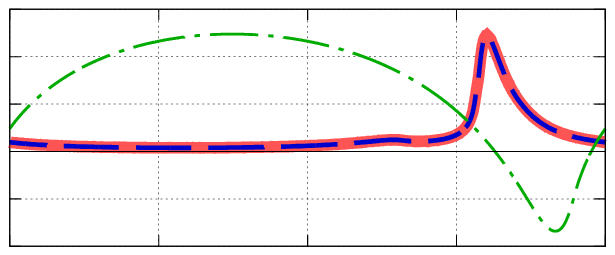}
	  \label{subfig:chi0-9-c1-3-c2-3}
	\end{minipage}
  }
  \\[0.2\baselineskip]
  \begin{minipage}{0.99\textwidth}
  \centering\include{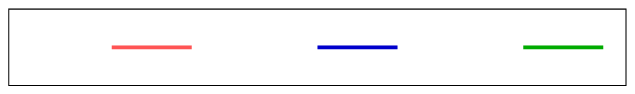}
\end{minipage}
\vspace{-0.3cm}
  \caption{Comparison of experimental
  ($f^{\emph{\scriptsize{exp}}}_{\triangle,\infty}$) and numerical
($f^{\emph{\scriptsize{num}}}_{\triangle,\infty}$) probability density
function. Additionally the respective inner integral $f^{\emph{\scriptsize{in}}}(\alpha)=\int_{l_H}^{u_H}g(\alpha,h)\cdot f_H(h)\di h$}
\label{fig:establishedparameters}
\end{figure}

In Figure~\ref{fig:conv3d} one can see the expected difference of consecutive
convergence indicators if $t$ tends to infinity. We present data for parameter
configurations $\chi\in[-1,1]$ and $c_l=c_g=c\in[-0.5,4.5]$. The sharp border
between the blue and the red area delimits the parameter configurations where
convergence can be expected. This delimitation is already displayed in
Figure~\ref{fig:convergencebounds}. In Table~\ref{tab:establishedparameters} for
an established set of parameters also the expected difference of consecutive
convergence indicators is presented. Figure~\ref{fig:establishedparameters}
also displays the corresponding stationary distribution on the angles in
comparison to an empirical probability density function, which is obtained by
experiments. Also the value of the inner integral
$f^{\emph{\scriptsize{in}}}(\alpha)=\int_{l_H}^{u_H}g(\alpha,h)\cdot f_H(h)\di
h$ is visualized.

To evaluate this values we used spline interpolation with initially $64$ equidistant knot
points and increased the number of knot points adaptively to $2048$ knot points
for both types of splines used to calculate the next probability density
function on the angles (see Section~\ref{sec:calcnextcdf}).
Additional knot points are placed where the third derivative of calculated
splines changes the most.
We use periodic boundary conditions for the splines representing
$F_{\triangle,t}$ in the sense that
$F'_{\triangle,t}(-\pi/2)=F'_{\triangle,t}(\pi/2)$,
$F''_{\triangle,t}(-\pi/2)=F''_{\triangle,t}(\pi/2)$ and obviously
$F'_{\triangle,t}(-\pi/2)=0\neq 1=F'_{\triangle,t}(\pi/2)$.
For the spline approximating $\Pr[f^+(\tan(\alpha_t),h)\leq\tan(\beta_i)]\cdot
f_H(h)$ for any value $\beta_i$ we split this function into several pieces
according to the cases for the function $f_H$ (see
Equation~\eqref{eq:densityH}) as we do not have continuous derivatives at these
positions. For each part we use natural boundary conditions, i.\,e., the second
derivative at boundary points is zero.
We stop doing iterations if the \LPnorm{2} of
$F_{\triangle,t}-F_{\triangle,t+1}$ is at most $10^{-7}$.
Please note the \LPnorm{2} of a function $f$ is
$\sqrt{\int_{\Omega}\|f(x)\|^2\di\mu(x)}$.
Here the respective space $\Omega$ is $[-\pi/2,\pi/2]$ and for the
measure $\mu$ we use the uniform distribution and consequently obtain
$\sqrt{\int_{-\pi/2}^{\pi/2}(F_{\triangle,t}(\alpha)-F_{\triangle,t+1}(\alpha))^2\cdot(1/\pi)\cdot\di\alpha}$
for the \LPnorm{2}.
Usually this procedure reaches this tolerance after few ($\approx 100-400$)
iterations, but especially in cases where $\vert \chi \vert$ is much larger
than $\vert c_l\vert$ and $\vert c_g\vert$ many iterations are required (for
$\chi=1,c_l=c_g=0.01$ we need $\approx 10\,000$ iterations).
To overcome periodic developments fast we keep a small portion of the old
distribution for the next distribution.

For evaluation of Equation~\eqref{eq:shorterexpectedxvdifference} we finally
used $8192$ knot points for both types of splines appearing in this task (see
Section~\ref{sec:calcfinal}). Here natural boundary conditions are used for
both types of splines, i.\,e., second derivative at boundary points is zero.
The only problematic case here is $\chi=0$ as then the inner integral
$f^{\emph{\scriptsize{in}}}$ is $-\infty$ at the borders. In all other cases
$f_{in}(\alpha)=\ln(2\cdot\chi^2)$ if $\alpha$ is equal to $-\pi/2$ or $\pi/2$.

Nevertheless for all evaluations the absolute errors are less than $10^{-4}$ on
the final result - this is even true for cases with $\chi=0$, which could be reached by
clipping the boundary by $10^{-15}$. In most cases the error is considerably
smaller ($\approx 10^{-7}$).

\subsection{Comparison to Empirical Results}

To supply further evidence that the obtained results are correct and no
conceptual errors are present we compare the results with observations on
experiments.

For experiments we analyzed a \textbf{single} PSO execution in a single
dimension with a huge number of iterations ($50\,000\,000$) and evaluated the
average difference of consecutive convergence indicators.
There are several reasons for using only a single PSO execution instead of
using multiple PSO executions.
If we use multiple executions we can take into account only the iterations
after the distribution of the current angle of the vector consisting of the
position $x$ and the velocity $v$ is similar to the stationary limit
distribution. Therefore we need to know how many iterations we have to make
until the result is useful but this number of iterations can not be determined
easily. Additionally for each PSO execution this first part of iterations costs
much computing power without benefit.
Using the average of all differences of consecutive convergence indicators has
the benefit that the share of the starting period becomes more and more
negligible the longer the experiment is running.
The covariance between differences of different iterations becomes smaller the
further apart the iterations are.
Therefore the average difference of consecutive convergence indicators measured
on this single PSO execution converges to a normal distribution (also called
Gaussian distribution).
By estimating the variance of single differences of consecutive convergence
indicators and the covariances of those also the variance of the average value
can be estimated.
While comparing the difference of the measured values and the numerically
calculated values inversely scaled by the square root of the estimated variance
among all evaluated test sets we observe characteristics of a normal
distribution as expected, i.\,e., the relative frequency of large deviations
conforms to quantiles of the normal distribution. Statistical tests on the
equivalence do not make sense as the numerically calculated values are not
exact. Especially for numerical results with larger errors (up to $10^{-4}$)
the estimated standard deviations are of the same magnitude as the numerical
errors.  Nevertheless, the experiments confirm the correctness of the numerical
evaluations. As a reference we refer to Figure~\ref{fig:establishedparameters},
where the numerically obtained limit distribution on the angles is compared
with the experimentally obtained relative frequencies. The relative frequencies
are evaluated for $1\,000$ blocks of equal size. The $i$th block captures the
relative frequency of the interval
$[-\pi/2+(i-1)\cdot\pi/1\,000,-\pi/2+i\cdot\pi/1\,000)$. If $n_i$ is the number
of iterations such that the respective angle is in the $i$th block then the
relative frequency is the number of iterations with an angle in that block
divided by the number of all iterations and divided by the block length:
$n_i\cdot 1\,000 / (\pi\cdot 50\,000\,000)$.
Figure~\ref{fig:establishedparameters} shows that for all visualized parameter sets
the numerically evaluated limit distribution coincides with the experimentally
evaluated relative frequencies up to some noise in the relative frequencies.

\section{Conclusion}
\label{sec:conclusion}

Parameter selection for meta-heuristics is a topic that is widely
regarded in the literature. In this paper we focus on the parameter
selection for PSO and its influence on the convergence of the particle
swarm. We introduce a new convergence indicator $\Phi_t$ that can be
used to prove for a selection of parameters whether the swarm will finally
converge or diverge. If in expectation the difference of two
consecutive convergence indicators is negative the swarm will finally
converge. We introduce a series of equations for calculating this
difference of convergence indicators and explain how to
numerically solve these equations using cubic spline interpolation
where the results have only minor errors.
Finally we provide experiments that confirm the correctness of
presented equations and their numerical evaluation.

\bibliographystyle{alpha}
\bibliography{literature}   

\newcommand{\GenerateBio}[3]{
\noindent
\begin{minipage}{0.23\textwidth}
\centering\includegraphics[width=1in,height=1.25in,clip,keepaspectratio]{#1}
\end{minipage}
\hfill
\begin{minipage}{0.765\textwidth}
\textbf{#2}
#3
\end{minipage}\\[\baselineskip]%
}

\end{document}